\documentclass[12pt,a4paper]{amsart}
\usepackage[T2A]{fontenc}
\usepackage[cp1251]{inputenc}
\usepackage[english]{babel}
\usepackage{amsmath,latexsym,amsthm,amsfonts,tipa}
\usepackage{amssymb,amscd,graphpap, stmaryrd}
\newcommand\BB{{\mathcal B}}

\newtheorem{theorem}{Theorem}[section]
\newtheorem{lemma}{Lemma}[section]
\newtheorem{corollary}{Corollary}[section]
\theoremstyle{definition}

\newcommand{\overdouble}[1]{\overline{\overline{#1}}}

\begin{document}

\title{Ultrafilters on metric Spaces}
\author{I.V. Protasov}
\date{}

\begin{abstract}
Let $X$ be an unbounded metric space, $B(x,r) = \{y\in X: d(x,y) \leqslant r\}$ for all $x\in X$ and $r\geqslant 0$. We endow $X$ with the discrete topology and identify the Stone-\v{C}ech compactification $\beta X$ of $X$ with the set of all ultrafilters on $X$. Our aim is to reveal some features of algebra in $\beta X$ similar to the algebra in the Stone-\v{C}ech compactification of a discrete semigroup \cite{b6}.

We denote $X^\# = \{p\in \beta X: \mbox{ each }P\in p\mbox{ is unbounded in }X\}$ and, for $p,q \in X^\#$, write $p\parallel q$ if and only if there is $r \geqslant 0$ such that $B(Q,r)\in p$ for each $Q\in q$, where $B(Q, r)=\cup_{x\in Q}B(x,r)$. A subset $S\subseteq X^\#$ is called invariant if $p\in S$ and $q\parallel p$ imply $q\in S$. We characterize the minimal closed invariant subsets of $X$, the closure of the set $K(X^\#) = \bigcup\{M : M\mbox{ is a minimal closed invariant subset of }X^\#\}$, and find the number of all minimal closed invariant subsets of $X^\#$. 

For a subset $Y\subseteq X$ and $p\in X^\#$, we denote $\bigtriangleup_p(Y) = Y^\# \cap \{q\in X^\#: p \parallel q\}$ and say that a subset $S\subseteq X^\#$ is an ultracompanion of $Y$ if $S = \bigtriangleup_p(Y)$ for some $p\in X^\#$. We characterize large, thick, prethick, small, thin and asymptotically scattered spaces in terms of their ultracompanions.

\end{abstract}

\maketitle

\section{Introduction}

Let $X$ be a discrete space, $\beta X$ be the Stone-\v{C}ech compactification of $X$. We take the points of $\beta X$ to be the ultrafilters on $X$, with the points of $X$ identified with the principal ultrafilters, so $X^*=\beta X\setminus X$ is the set of all free ultrafilters. The topology of $\beta X$ can be defined by stating that the set of the form $\overline{A} = \{p\in \beta X: A\in p\}$, where $A$ is a subset of $X$, are base for the open sets. The universal property of $\beta X$ states that every mapping $f:X\to Y$, where $Y$ is a compact Hausdorff space, can be extended to the continuous mapping $f^\beta: \beta X \to Y$.

If $S$ is a discrete semigroup, the semigroup multiplication has a natural extension to $\beta S$, see \cite[Chapter 4]{b6}. The compact right topological semigroup $\beta S$ has a plenty of applications to combinatorics, topological algebra and functional analysis, see \cite{b3}, \cite{b5}, \cite{b6}, \cite{b21}, \cite{b22}.

Now let $(X,d)$ be a metric space, $B(x,r) = \{y\in X: d(x,y)\leqslant r\}$ for all $x\in X$ and $r\geqslant 0$. A subset $V$ of $X$ is {\it bounded} if $V\subseteq B(x,r)$ for some $x\in X$ and $r\geqslant 0$. We suppose that $X$ is {\bf unbounded}, endow $\beta X$ with the {\bf discrete} topology and, for a subset $Y$ of $X$, put

$$Y^\# = \{p\in \beta X:\mbox{ each }P\in p\mbox{ is unbounded in }X\},$$
and note that $Y^\#$ is closed in $\beta X$.

For $p,q \in X^\#$, we write $p\parallel q$ if and only if there is $r\geqslant 0$ such that $B(Q,r)\in p$ for each $Q\in q$, where $B(Q, r)=\cup_{x\in Q}B(x,r)$. The parallel equivalence was introduced in \cite{b8} in more general context of balleans and used in \cite{b1}, \cite{b10}, \cite{b11}, \cite{b12}, \cite{b13}.

For $p\in X^\#$, we denote $\overdouble{p} = \{q\in X^\#: p \parallel q\}$ and say that a subset $S$ of $X^\#$ is invariant if $\overdouble{p}\subseteq S$ for each $p\in S$. Every nonempty closed invariant subset of $X^\#$ contains some non-empty minimal (by inclusion) closed invariant subset. We denote
$$K(X^\#) = \bigcup\{M: M\mbox{ is a minimal closed invariant subset of }X^\#\}.$$

After short technical Section 2, we show in Section 3 how one can detect whether $S\subseteq X^\#$ is a minimal closed invariant subset, and whether $q\in X^\#$ belongs to the closure of $K(X^\#)$ in $X^\#$. We prove that the set  of all minimal closed invariant subsets of $X^\#$ has cardinality $2^{2^{asden\,X}}$, where $asden\, X = \min\{|L|: L\subseteq X\mbox{ and }X=B(L,r)\mbox{ for some }r\geqslant 0\}$.

In Section 5 we show that from the ballean point of view the minimal closed invariant subsets are counterparts of the minimal left ideal in $\beta G$, where $G$ is a discrete group. Thus, the results of Section 3 are parallel to Theorems 4.39, 4.40 and 6.30(1) from \cite{b6}.

In Section 3 we use the following classification of subsets of a metric space. We say that a subset $Y$ of $X$ is
\begin{itemize}
\item {\it large} if $X=B(Y,r)$ for some $r\geqslant 0$;
\item {\it thick} if, for every $r\geqslant 0$, there exists $y\in Y$ such that $B(y,r)\subseteq Y$;
\item {\it prethick} if $B(Y,r)$ is thick for some $r\geqslant 0$;
\item {\it small} if $X\setminus Y \cap L$ is large for every large subset $L$ of $X$;
\item {\it thin} if, for each $r\geqslant 0$, there exists a bounded subset $V$ of $X$ such that $B(y,r)\cap Y = \{y\}$ for each $y\in Y\setminus V$.
\end{itemize}

We note that $Y$ is small if and only if $Y$ is not prethick \cite[Theorem 11.1]{b15}, and the family of all small subsets of $X$ is an ideal in the Boolean algebra of all subsets of $X$ \cite[Theorem 11.2]{b15}. Hence, for every finite partition of $X$, at least one cell is prethick.

For $p\in X^\#$ and $Y\subseteq X$, we put
$$\displaystyle\bigtriangleup\nolimits_p(Y) = \overdouble{p}\cap Y^\#,$$
and say that $\bigtriangleup_p(Y)$ is a {\it $p$-companion} of $Y$. A subset $S\subseteq X^\#$ is called an {\it ultracompanion} of $Y$ if $S = \bigtriangleup_p(Y)$ for some $p\in X^\#$.

In Section 4 we characterize all above defined subsets of $X$ and asymptotically scattered subsets from \cite{b14} in terms of their ultracompanions.
	
During the exposition, $X$ is an {\bf unbounded metric space}.

\section{Parallelity}

\begin{lemma} Let $Y$ be a subset of $X$, $r\geqslant 0$, $q\in X^\#$. If $B(Y,r)\in q$ then there exists $s\in Y^\#$ such that $q\parallel s$.
\end{lemma}
\begin{proof} For each $Q\in q$, we denote $S_Q = B(Q,r)\cap Y$. Since $B(Y,r) \in q$, we have $S_Q \ne \varnothing$. The family $\{S_Q: Q\in q\}$ is contained in some ultrafilter $s\in Y^\#$. Clearly, $s\parallel q$.
\end{proof}

\begin{lemma} The closure $cl\, S$ in $X^\#$ of an invariant subset $S$ of $X^\#$ is invariant.
\end{lemma}
\begin{proof} Let $p\in cl\, S$, $q\in X^\#$ and $q\parallel p$. We take $r\geqslant 0$ such that $B(Q, r)\in p$ for each $Q\in q$. Then we fix $Q\in q$ and pick $p'\in S$ such that $B(Q,r)\in p'$. By Lemma 2.1, there is $q' \in Q^\#$ such that $q'\parallel p'$. It follows that $q\in cl\, S$.
\end{proof}

\begin{lemma} Each nonempty closed invariant subset $S$ of $X^\#$ contains a nonempty minimal closed invariant subset.
\end{lemma}
\begin{proof} We observe that the intersection of any family of invariant subsets is invariant, use compactness of $X^\#$ and apply the Zorn's lemma.
\end{proof}

\begin{lemma} A closed subset $M$ of $X^\#$ is minimal if and only if $M\ne\varnothing$ and $M=cl\, \overdouble{p}$ for each $p\in M$.
\end{lemma}
\begin{proof} Let $M$ be minimal and $p\in M$. Since $\overdouble{p}$ is invariant, by Lemma 2.2, $cl\, \overdouble{p}$ is invariant so $M=cl\, \overdouble{p}$ by minimality of $M$.

On the other hand, let $M=cl\, \overdouble{p}$ for each $p\in M$. Let $S$ be a closed invariant subset of $M$. We take $p\in S$. Since $cl\, \overdouble{p} \subseteq S$ and $cl\, \overdouble{p} = M$, we have $M=S$ so $M$ is minimal.
\end{proof}

Recall that a function $f:X\times X\to \mathbb{R}^+$ is a {\it semi-metric} if $f$ is symmetric and satisfies the triangle inequality (but $f(x,y) = 0$ does not necessarily imply $x=y$). For each $p\in X^\#$, the metric $d$ on $X$ defines the semi-metric $d_p^\#$ on $\overdouble{p}$ by 
$$d_p^\#(s,q) = \inf\{r\geqslant 0: B(S,r)\in q\mbox{ for each }S\in s\}.$$
If the spectrum $spec(X,d) = \{d(x,y): x,y\in X\}$ is discrete then $d_p^\#$ is a metric.

\section{Minimality}

\begin{theorem} For every ultrafilter $p\in X^\#$ the following two statements are equivalent
\begin{itemize}
\item[{\it (i)}] $cl\, \overdouble{p}$ is a minimal closed invariant subset of $X^\#$;
\item[{\it (ii)}] for every $P\in p$, there exists $r\geqslant 0$ such that $\overdouble{p}\subseteq (B(P,r))^\#$.
\end{itemize}
\end{theorem}
\begin{proof} $(i) \Rightarrow (ii)$. Suppose the contrary: there is $P\in p$ such that, for each $r\geqslant 0$, $\overdouble{p}\setminus(B(P,r))^\# \ne \varnothing$. We consider a filter $\varphi$ on $X$ with the base $\{X\setminus B(P,r): r\geqslant 0\}$ and choose $q\in cl\, p$ such that $\varphi \subseteq q$ so $X\setminus B(P,r)\in q$ for each $r\geqslant 0$.

We take an arbitrary $s\in \overdouble{q}$ and assume that $P\in s$. Since $s\parallel q$, there is $r_0\geqslant 0$ such that $B(P, r_0)\in q$ contradicting the choice of $q$. Hence $P\notin s$ and $p\notin cl\, \overdouble{q}$. It follows that $cl\, \overdouble{p} \subset cl\, \overdouble{q}$ and, by Lemma 2.1, $cl\, \overdouble{p}$ is not minimal.

$(ii) \Rightarrow (i)$. We take an arbitrary $q\in cl\,\overdouble{p}$. By $(ii)$, for every $P\in p$, there is $r(P)\geqslant 0$ such that $q\in B(P,r(P))^\#$. Using Lemma 2.1, we choose $q(P)\in P^\#$ such that $q\parallel q(P)$. Since $P\in q(P)$ for every $P\in p$, the net $\{q(P): P\in p\}$ converges to $p$. Hence $p\in cl\,\overdouble{q}$ and, by Lemma 2.4, $cl\,\overdouble{p}$ is minimal.
\end{proof}

\begin{theorem} For $q\in X^\#$, $q\in cl\, K(X^\#)$ if and only if each $Q\in q$ is prethick.
\end{theorem}
\begin{proof} Let $q\in cl\, K(X^\#)$ and $Q\in q$. We pick $p\in K(X^\#)$ such that $Q\in p$. By Theorem 3.1, there exists $r\geqslant 0$ such that $\overdouble{p}\subseteq(B(Q,r))^\#$. Hence $\overdouble{p} = \bigtriangleup_p(B(Q,r))$ and, by Theorem 4.2, $Q$ is prethick.

On the other hand, let each $Q\in q$ be prethick. By Theorem 4.2, there exist $p\in X^\#$ and $r\geqslant 0$ such that $\overdouble{p} = \bigtriangleup_p(B(Q,r))$ so $\overdouble{p}\subseteq(B(Q,r))^\#$ and $cl\,\overdouble{p} = (B(Q,r))^\#$. The closed invariant subset $cl\,\overdouble{p}$ contains some  minimal closed invariant subset $cl\,\overdouble{p_1}$. Since $B(Q,r)\in p_1$, by Lemma 2.1, there is $p_2\in X^\#$ such that $Q\in p_2$ and $p_2\parallel p_1$. Since $cl\,\overdouble{p_2}$ is minimal, $q\in cl\,K(X^\#)$.
\end{proof}

Following \cite{b8}, we denote by $\sim$ the smallest (by incusion) closed (in $X^\#\times X^\#$) equivalence on $X^\#$ such that $\parallel \subseteq \sim$. The quotient $X^\#/\sim$ is a compact Hausdorff space. It is called the {\it corona} $X$ and is denoted by $\check X$. For each $p\in X^\#$, we denote $\check p = \{q\in X^\#: q\sim p\}$. To clarify the virtual equivalence $\sim$, we use the slowly oscillating functions. 

A function $h:(X,d)\to [0,1]$ is called slowly oscillating if, for any $n\in \omega$ and $\varepsilon > 0$, there exists a bounded subset $V$ of $X$ such that, for every $x\in X\setminus V$,
$$diam\, h(B(x,n))<\varepsilon,$$
where $diam\, A=\sup\{|x-y|:x,y\in A\}$. By \cite[Proposition 1]{b10}, $p\sim q$ if and only if $h^\beta(p) = h^\beta(q)$ for every slowly oscillating function $h:X\to [0,1]$.

\begin{theorem} For every unbounded metric space $X$, we have $|\check X| = asden\, X$, where $asden\, X = \min\{|L|: L \mbox{ is a large subset of }X\}$.
\end{theorem}
\begin{proof} We take a large subset $L$ of $X$ such that $|L| = asden\, X$. Given any $q\in X^\#$, we use Lemma 2.1 to choose $p\in L^\#$ such that $q\parallel p$. It follows that  $|\check X|\leqslant 2^{2^{|Y|}}$, the number of all ultrafilters on $Y$, so $|\check X|\leqslant 2^{2^{asden\, X}}$.

We check that$|\check X|\geqslant 2^{2^{asden\, X}}$. By \cite[Theorem 2.3]{b4}, there is a thin subset $Y$ of $X$ such that $|Y| = asden\, X$ and $|Y|_X = |Y|$ where $|Y|_X = \min\{|Y\setminus V|: V\mbox{ is a bounded subset of }X\}$. We fix $x_0\in X$ and note that $Y^\# = \cap_{n\in\omega}\overline{(Y\setminus B(x_0,n))}$. Since $|Y\setminus B(x_0,n)| = |Y|$ for each $n\in \omega$, we have $|Y^\#| = 2^{2^|Y|}$. Let $p,q$ be distinct ultrafilters from $Y^\#$. We partition $Y\setminus V = P\cup Q$ so that $P\in p$, $Q\in q$. Since $Y$ is thin, $B(P,n)\cap B(Q,n)$ is bounded for each $n\in \omega$. By \cite[Lemma 4.2]{b8}, there is a slowly oscillating function $h:X\to [0,1]$ such that $h^\beta(p) \ne h^\beta(q)$. By \cite[Proposition 1]{b10}, $\check p \ne \check q$. Hence $|\check X| \geqslant |Y^\#| = 2^{2^{asden\, X}}$.
\end{proof}

\begin{corollary} For every unbounded metric space $X$, the set $\mathcal{M}$ of all minimal closed invariant subsets of $X$ has cardinality $2^{2^{asden\, X}}$.
\end{corollary}
\begin{proof} For each $p\in X^\#$, closed invariant subset $\check p$ contains some member of $\mathcal{M}$ so, by Theorem 3.3, $|\mathcal{M}|\geqslant 2^{2^{asden\, X}}$. To verify the converse inequality, we take a large subset $L$ of $X$ such that $|L| = asden\, X$ and apply Lemma 2.1.
\end{proof}

Let $p$, $q$ be ultrafilters from $X^\#$ such that $\overdouble{p}, \overdouble{q}$ are countable and $cl\, \overdouble{p} \cap cl\, \overdouble{q} \ne \varnothing$. Applying the Frolik lemma \cite[Theorem 3.40]{b6}, we conclude that either $cl\, \overdouble{p} \subseteq cl\, \overdouble{q}$ or $cl\, \overdouble{q} \subseteq cl\, \overdouble{p}$. I do not know if this statement holds for arbitrary $p,q\in X^\#$, such that $cl\, \overdouble{p} \cap cl\, \overdouble{q} \ne \varnothing$.

\section{Ultracompanions}

\begin{theorem}\label{t_4_1} A subset $Y$ of $X$ is large if and only if $\bigtriangleup_p(Y) \ne \varnothing$ for every $p\in X^\#$.
\end{theorem}
\begin{proof} Suppose that $Y$ is large and pick $r\geqslant 0$ such that $X=B(Y,r)$. We take an arbitrary $p\in X^\#$. Since $B(Y,r)\in p$, by Lemma 2.1, there is $q\in Y^\#$ such that $p\parallel q$ so $q\in \bigtriangleup_p(Y)$ and $\bigtriangleup_p(Y)\ne \varnothing$.

Assume that $\bigtriangleup_p(Y) \ne \varnothing$ for each $p\in X^\#$. Given any $p\in X^\#$, we choose $r_p\geqslant 0$ such that $B(Y,r_p)\in p$. Then we consider a covering of $X^\#$ by the subsets $\{(B(Y,r_p))^\#: p\in X^\#\}$ and choose its finite subcovering $(B(Y,r_{p_1}))^\#, \ldots, (B(Y,r_{p_n}))^\#$. Let $r=\max\{r_{p_1}, \ldots, r_{p_n}\}$. Since $X\setminus B(Y,r)$ is bounded, we can choose $r' \geqslant r$ such that $X=B(Y,r')$ so $Y$ is large.
\end{proof}

\begin{theorem} For a subset $Y$ of $X$, the following statements hold
\begin{itemize}
\item[{\it (i)}] $Y$ is thick if and only if there exists $p\in X^\#$ such that $\overdouble{p} = \bigtriangleup_p(Y)$;
\item[{\it (ii)}] $Y$ is prethick if and only if there exists $p\in X^\#$ and $r\geqslant 0$ such that $\overdouble{p} = \bigtriangleup_p(B(Y,r))$;
\item[{\it (iii)}] $Y$ is small if and only if, for every $p\in X^\#$ and each $r\geqslant 0$, $\overdouble{p} \ne \bigtriangleup_p(B(Y,r))$.
\end{itemize}
\end{theorem}
\begin{proof} $(i)$. We note that $Y$ is thick if and only if $X\setminus Y$ is not large and apply Theorem \ref{t_4_1}.

$(ii)$ follows from $(i)$.

$(iii)$ Remind that $Y$ is small if and only if $Y$ is not prethick and apply $(ii)$.
\end{proof}

For a natural number $n$, a subset $Y$ of $X$ is called {\it $n$-thin} if, for each $r\geqslant 0$, there exists a bounded subset $V$ of $X$ such that $|B(y,r)\cap Y|\leqslant n$ for each $y\in Y\setminus V$. Clearly, $Y$ is thin if and only if $Y$ is 1-thin.

\begin{theorem}
For a subset $Y$ of $X$ and a natural number $n$, the following statements are equivalent
\begin{itemize}
\item[{\it (i)}] $Y$ is $n$-thin;
\item[{\it (ii)}] $Y$ can be partitioned into $\leqslant n$ thin subsets;
\item[{\it (iii)}] $|\bigtriangleup_p(Y)|\leqslant n$ for each $p\in X^\#$.
\end{itemize}
\end{theorem}
\begin{proof}
The equivalence $(i)\Leftrightarrow(ii)$ was proved in \cite[Theorem 2.1]{b7}. To verify $(ii)\Leftrightarrow(iii)$, it suffices to show that $Y$ is thin if and only if $|\bigtriangleup_p(Y)|\leqslant 1$ for each $p\in X^\#$.

Let $Y$ be thin, $Y\in p$ and $q\in \bigtriangleup_p(Y)$. We choose $r\geqslant 0$ such that $B(P,r)\in q$ for each $P\in p$. Then we pick a bounded subset $V$ of $X$ such that $B(y,r)\cap Y = \{y\}$ for each $y\in Y\setminus V$. Let $P\in p$ and $P\subseteq Y\setminus V$. Since $Y\in q$ and $B(P,r)\in q$, we have $P\in q$. Hence $q=p$.

Suppose that there are two distinct parallel ultrafilters $p,q$ in $Y^\#$. We take $r\geqslant 0$ such that $B(P,r)\in q$ for each $P\in p$. Then we choose $P_0\in p$, $Q_0\in q$ such that $P_0\subset Y$, $Q_0\subset Y$ and $P_0\cap Q_0 = \varnothing$. We note that $\{x\in P_0: B(x,r)\cap Q_0\}\in q$. Since this set is unbounded and $P_0\cap Q_0 = \varnothing$, we conclude that $Y$ is not thin.
\end{proof}

Following \cite{b2}, we say that a subset $Y$ of a metric space $X$ {\it has asymptotically isolated $m$-balls}, $m\in\omega$ if, for every $n\in \omega$, there exists $n' > n$ and $y\in Y$ such that $B(y,n')\setminus B(y,m) = \varnothing$. If $Y$ has asymptotically isolated $m$-balls for some $m\in \omega$, we say that $Y$ {\it has asymptotically isolated balls}.

A metric space is called {\it asymptotically scattered} \cite{b14} if each unbounded subset of $X$ has asymptotically isolated balls. We say that a subset $Y$ of $X$ is asymptotically scattered if the metric space $(Y, d_Y)$ is asymptotically scattered.

Given $p\in X^\#$, we say that a subset $S\subseteq \overdouble{p}$ is {\it ultrabounded} with respect to $p$ if there is $r\geqslant 0$ such that, for each $q\in S$ and every $Q\in q$, we have $B(Q,r)\in p$. If $S$ is ultrabounded with respect to $p$ then $S$ is ultrabounded with respect to any $q\in \overdouble{p}$ so we shall say that $S$ is ultrabounded in $\overdouble{p}$. In other words, $S$ is ultrabounded if $S$ is bounded in the semi-metric space $(\overdouble{p}, d_p^\#)$.

\begin{theorem}
A subset $Y$ of $X$ is asymptotically scattered if and only if, for every $p\in X^\#$, the subset $\bigtriangleup_p(Y)$ is ultrabounded in $\overdouble{p}$.
\end{theorem}
\begin{proof}
Suppose that $Y$ is not asymptotically scattered and show how to find $p\in Y^\#$ such that $\bigtriangleup_p(Y)$ is not ultrabounded in $\overdouble{p}$. The proof of Theorem 3 from \cite{b14} gives us a sequence $(\{z_{n0}, z_{n1}, \ldots, z_{nn} \})_{n\in \omega}$ of subsets of $Y$ and an increasing sequence $(k_n)_{n\in \omega}$ in $\omega$ such that 
\begin{itemize}
\item[(1)] $d(z_{00}, z_{n0}) \geqslant n, n\in \omega$
\item[(2)] $k_i \leqslant d(z_{n0},z_{ni})\leqslant k_{i+1}$, $0 < i \leqslant n < \omega$;
\item[(3)] $d(x,y)\geqslant k_i$ for all $k\in Z_0 = \{z_{00}, z_{10}, \ldots\}$, $y\in Z_i = \{z_{ii}, z_{i+1\,i}, z_{i+2\,i}, \ldots\}$, $0 < i < \omega$.
\end{itemize}
By (1), the set $Z_0$ is unbounded in $X$. We fix an arbitrary $p\in Z_0^\#$. For each $i >0$, we use (2) and Lemma 2.1 to find $p_i\in Z_i^\#$ such that $p_i\parallel p$. Since $\{p_i: 0 < i < \omega\} \subseteq \bigtriangleup_p(Y)$, by (3), $\bigtriangleup_p(Y)$ is not ultrabounded in $\overdouble{p}$.

We suppose that $\bigtriangleup_p(Y)$ is not ultrabounded in $\overdouble{p}$ for some $p\in Y^\#$ and define a function $d_p: \overdouble{p} \to \mathbb{R}^+$ by the rule

$$d_p(q) = \inf\{r\in \mathbb{R}^+: B(Q,r)\in p \mbox{ for each }Q\in q\}.$$
Since $\bigtriangleup_p(Y)$ is not ultrabounded, one can choose a sequence $(q_n)_{n\in \omega}$ in $\bigtriangleup_p(Y)$ and an increasing sequence $(k_n)_{n\in \omega}$ in $\omega$ such that 
\begin{itemize}
\item[(4)] $k_n < d_p(q_n) < k_{n+1}$, $n\in \omega$.
\end{itemize}
Using (4), we choose a decreasing sequence $(P_n)_{n\in \omega}$ in $p$ and a sequence $(Q_n)_{n\in \omega}$, $Q_n\in q_n$ such that
\begin{itemize}
\item[(5)] $P_n \subset Y, Q_n \subset Y$, $n\in \omega$.
\item[(6)] $k_n \leqslant d(x,y) \leqslant k_{n+1}$ for all $x\in P_n$, $y\in Q_n$.
\end{itemize}
For each $n\in \omega$, we pick $z_n\in P_n$ such that the set $Z = \{z_n: n\in \omega\}$ is unbounded in $X$. By (5) and (6), $Z$ has no isolated balls in $(Y,d_Y)$. Hence $Y$ is not asymptotically scattered.
\end{proof}

\section{Ballean context}

Following \cite{b18}, we say that a {\em ball structure} is a triple $\BB=(X,P,B)$, where $X$, $P$ are non-empty sets and, for every $x\in X$
and $\alpha\in P$, $B(x,\alpha)$ is a subset of $X$ which is called a {\em ball of radius $\alpha$} around $x$. It is supposed that $x\in B(x,\alpha)$ for all $x\in X$ and $\alpha\in P$. The set $X$ is called the {\em support} of $\BB$, $P$ is called the set of {\em radii}.

Given any $x\in X$, $A\subseteq X$, $\alpha\in P$, we set
$$B^\ast(x.\alpha)=\{y\in X: x\in B(y,\alpha)\},\text{ } B(A,\alpha)=\bigcup_{a\in A}B(a,\alpha).$$
A ball structure $\BB=(X,P,B)$ is called a {\em ballean} if

$\bullet$ for any $\alpha,\beta\in P$, there exist $\alpha',\beta'\in P$ such that, for every $x\in X$,
$$B(x,\alpha)\subseteq B^\ast(x,\alpha'),\text{ }B^\ast(x,\beta)\subseteq B(x,\beta');$$

$\bullet$ for any $\alpha,\beta\in P$, there exists $\gamma\in P$ such that, for every $x\in X$,
$$B(B(x,\alpha),\beta)\subseteq B(x,\gamma).$$

$\bullet$ for any $x,y \in X$, there is $\alpha\in P$ such that $y\in B(x,\alpha)$.

A ballean $\BB$ on $X$ can also be determined in terms of entourages of the diagonal $\Delta_X$ of $X\times X$ (in this case it is called a {\em coarse structure} \cite{b20}) and can be considered as an asymptotic counterpart of a uniform topological space.

Let $\BB_1=(X_1,P_1,B_1)$ and $\BB_2=(X_2,P_2,B_2)$ be balleans. A mapping $f:X_1\to X_2$ is called a $\prec$-mapping if, for every $\alpha\in P_1$, there exists $\beta\in P_2$ such that, for every $x\in X$ $f(B_1(x,\alpha))\subseteq B_2(f(x),\beta)$. A bijection $f: X_1\to X_2$ is called an {\it asymorphism} if $f, f^{-1}$ are $\prec$-mappings.

Every metric space $(X,d)$ defines the ballean $(X,\mathbb{R}^+,B_d),$ where
$B_d(x,r)=\{y\in X:d(x,y)\leqslant r\}$. For criterion of metrizability of balleans see \cite[Theorem 2.1.1]{b18}.

Let $G$ be a group with the identity $e$. We denote by $\mathcal{F}_G$ the family of all finite subsets of $G$ containing $e$ and get the group ballean $\BB(G) = (G, \mathcal{F}_G, B)$ where $B(g,F) = Fg$ for all $g\in G$, $F\in \mathcal{F}_G$.

We observe that all definitions in this paper do not use the metric on $X$ but only balls so can be literally rewritten for any balleans in place of metric spaces. Moreover, a routine verification ensures that Theorems 3.1, 3.2, 4.1, 4.2 remains true for any balleans.

Theorem 3.3 and Corollary 3.4 fail to be true for the following ballean from \cite{b19}. Let $\kappa$ be an infinite cardinal, $\varphi$ be a free ultrafilter on $\kappa$ such that $|\Phi| = \kappa$ for each $\Phi \in \varphi$. We consider the ballean $B(\kappa, \varphi) = (\kappa, \varphi, B)$ where $B(x, \Phi) = \{x\}$ if $x\in \Phi$, and $B(x, \Phi) = X\setminus \Phi$ if $x\notin \Phi$. Then $\check X$ is the singleton $\{\varphi\}$.

G.~Bergman constructed (see \cite{b16}) a group $G$ of cardinality $\aleph_2$ containing a 2-thin subset which cannot be partitioned into two thin subsets. Hence Theorem 4.3 does not hold for the group ballean $\BB(G)$.

I do not know whether Theorem 4.4 can be generalized to all balleans.

Coming back to the group ballean, we note that $G^\# = G^*$ and, for $p,q\in G^*$, $p\parallel q$ if and only if $q = gp$ for some $g\in G$. Hence $\overdouble{p} = Gp$, $cl\,\overdouble{p} = (\beta G)p$ is the principal left ideal in $\beta G$ generated by $p$, and the minimal closed invariant subsets in $G^\#$ are exactly the minimal left ideals of the semigroup $\beta G$. For ultracompanion's characterizations of subsets of a group see \cite{b17}.

A metric space is called {\it uniformly locally finite} if, for each $n\in \omega$, there is $c(n)\in\omega$ such that $|B(x,n)|\leqslant c(n)$ for each $x\in X$. By \cite[Theorem 1]{b9} there exists a group $G$ of permutations of $X$ such that the metric ballean is asymorphic to the ballean $\BB(G, X) = (X, \mathcal{F}_G, B)$, where $B(x,F) = Fx$ for all $F\in \mathcal{F}_G$, $x\in X$. The action of $G$ on $X$ can be extended to the action of $G$ on $\beta X$ and then to the action of $\beta G$ on $\beta X$. If $p,q\in X^\#$ and $q\parallel p$ then $q = gp$ for some $g\in G$ so $cl\,\overdouble{p} = (\beta G)p$.

Department of Cybernetics

Kyiv University

Volodimirska 64

01033 Kyiv

Ukraine

i.v.protasov@gmail.com 
\end{document}